\documentclass[10pt]{article}
\usepackage{color}
\usepackage{amssymb}
\usepackage{amsthm,array,amssymb,amscd,amsfonts,latexsym, url}
\usepackage{amsmath}
\usepackage[all]{xy}
\newtheorem{theo}{Theorem}[section]
\newtheorem{prop}[theo]{Proposition}

\newtheorem{lemm}[theo]{Lemma}
\newtheorem{coro}[theo]{Corollary}
\newtheorem{question}[theo]{Question}
\newtheorem{rema}[theo]{Remark}
\newtheorem{Defi}[theo]{Definition}

\newtheorem{say}[theo]{}

\renewcommand{\a}[0]{{\mathbb A}}
\newcommand{\qtq}[1]{\quad\mbox{#1}\quad}
\newcommand{\cbs}[0]{\operatorname{CBS}}
\date{}
\title{Flat pushforwards of Chern classes and   the  smoothability of  cycles below the middle dimension}
\author{J\'anos Koll\'ar\footnote{Partial  financial support   was provided  by  the NSF under grant number  DMS-1901855}, Claire Voisin\footnote{The author is supported by the ERC Synergy Grant HyperK (Grant agreement No. 854361).}}

\begin{document}
\maketitle

\begin{abstract} We prove in this paper the smoothability of cycles modulo rational equivalence below the middle dimension, that is, when the dimension is strictly smaller than the codimension. We  introduce and study the class  of cycles obtained as  ``flat pushforwards of Chern classes" (or equivalently, flat pushforwards of products of divisors) and prove that they are smoothable below the middle dimension.  Our main result is  that all cycles (of any dimension) on a smooth projective variety are flat pushforwards of Chern classes. In the case of abelian varieties, one can even restrict to smooth pushforwards of Chern classes.
 \end{abstract}
 \section{Introduction}
  Let $X$ be a smooth projective variety of dimension $n$. Following \cite{hartshorne}, we will say that a cycle class $z\in{\rm CH}_d(X)$ is smoothable if it belongs to the subgroup of  $\rm{CH}_d(X)$  generated by the classes of  $d$-dimensional smooth subvarieties of $X$. A number of non-smoothability results have been proved above the middle dimension, that is, when $2d\geq n$, since the question of smoothability was first  raised by Borel and Haefliger \cite{borelhaefliger} for cohomology classes. Hartshorne, Rees  and Thomas \cite{hartshorne} proved the  non-smoothability of the second Chern class $c_2(\mathcal{E})$ of the tautological subbundle $\mathcal{E}$ on a Grassmannian $G(3,n),\,n\geq6$. Debarre \cite{debarreabvar} proved the non-smoothability of the minimal class $\frac{\theta^2}{2}$ on a very general Jacobian of curve of genus $\geq 7$, where $\theta $ is the class of a Theta-divisor. This class of  examples has been greatly expanded in \cite{BenoistDebarre}. Benoist \cite{benoist}  exhibits examples of nonsmoothable $d$-cycles on varieties of dimension $n$ for many possible pairs $(d,n)$ above the middle dimension, including the case  where $2d=n$, under some arithmetic condition on the codimension $c=n-d$. When $2d\geq n+2$, a big question which remains completely open despite these counterexamples concerns the smoothability of cycles with $\mathbb{Q}$-coefficients.

   When $2d-1\leq n$,  Kleiman proves in  \cite{kleiman}  that
  for any cycle $z\in {\rm CH}_d(X)$ of codimension $c$, the cycle $(c-1)!z$ is smoothable. (A similar result in the range $2d<n$ already appeared  in \cite{hironaka}, but  Hironaka  acknowledges there  the help of Kleiman.) We also mention a related result by Sumihiro \cite[Theorem 3.2]{sumihiro} saying that for any cycle $z\in {\rm CH}^c(X)$, the cycle $(c-1)!z$ becomes  smoothable after pull-back under a flat base-change $X'\rightarrow X$, with $X'$ smooth. (Note that there is no condition on $c$ in this statement.)
For the cycles themselves (as opposed to a multiple), Hironaka  proved in 1968 the following result.
 \begin{theo}\label{theohironaka} (Hironaka \cite{hironaka}) Cycles of dimension $d\leq 3$ are smoothable on smooth varieties of dimension $n>2d$.
 \end{theo}
  We study in this paper  the smoothability problem for  cycles modulo rational equivalence below the middle dimension, which is mentioned in the introduction of \cite{benoist}, and our first main result is
 \begin{theo}\label{conjsmoothing}  Let $X$ be a smooth  projective variety of dimension $n$, defined over a field of characteristic $0$. Then for any integer  $d$ such that $2d<n$, any cycle $z\in \rm{CH}_d(X)$ is smoothable. \end{theo}

  Note that this result is related to   the easy case of  Whitney's embedding  theorem in differential topology, thanks to Hironaka's resolution theorem  \cite{hironakareso}. Indeed, for any $d$-dimensional subvariety $Z\subset X$, we can resolve the singularities of $Z$ and get a proper  morphism $\tilde{j}:\widetilde{Z}\rightarrow X$ such that $\tilde{j}_*[\widetilde{Z}]=[Z]$ in ${\rm CH}_d(X)$.  Our statement  is that, if $2d<n$, we can replace modulo rational equivalence $\tilde{j}:\widetilde{Z}\rightarrow X$ by an integral combination of embeddings $j_i:Z_i\hookrightarrow X$ of smooth subvarieties. The difficulty is the following : We can of course  embed $\widetilde{Z}$ in $X\times \mathbb{P}^n$ over $X$, and then, by an easy projection argument, we see that it suffices to construct a cycle $Z'=\sum_in_iZ'_i$  in $X\times \mathbb{P}^n$, which is rationally equivalent to $\widetilde{Z}$, and such that the $Z'_i$ are both smooth and in general position. The Chow moving lemma provides such a cycle in general position but unfortunately the $Z'_i$ are not smooth starting from dimension $4$.

\begin{rema}{\rm In \cite[Theorem 0.3]{benoist}, Benoist produces,  for infinitely many values of $d$, examples of cycles of dimension $d$ and codimension $d$ on smooth projective varieties over $\mathbb{C}$,  that are not smoothable. Theorem \ref{conjsmoothing} shows that, at least for these values of $d$,  his examples are optimal and   the  condition $2d<n$ is necessary for smoothability.}
\end{rema}

Theorem \ref{conjsmoothing} is obtained as a consequence of a more general structure result for algebraic cycles of any dimension  that we now describe. For any   smooth variety $X$, we denote by
 ${\rm CH}^*(X)_{\rm  Ch}\subset {\rm CH}^*(X)$ the subring generated by Chern classes of vector bundles on $X$.  The standard formula (\ref{eqformulap1fact}) combined with locally free resolutions shows that $$(c-1)!{\rm CH}^c(X)\subset  {\rm CH}^c(X)_{\rm Ch}.$$ In  particular we have  ${\rm CH}^*(X)\otimes \mathbb{Q}={\rm CH}^*(X)_{\rm Ch}\otimes \mathbb{Q}$.
 However, it is well-known that ${\rm CH}^*(X)_{\rm Ch}$ can be a proper subring of ${\rm CH}^*(X)$. We refer to  \cite{debarre} for an explicit example and to Lemma \ref{leexample} for another example. As discussed in Sections \ref{secexamples}, further examples are provided by homogeneous varieties $G/H$, where $G$ is  a semi-simple, simply connected group, $H$ is a Borel subgroup and the torsion order of $G$ is not $1$ (we are grateful to Burt Totaro for explaining this to us).
 Cycles of Chern type, that is, elements of  ${\rm CH}_d(X)_{\rm Ch}$, are relevant for our subject, since we know that they  are smoothable under the Whitney condition $2d<{\rm dim}\,X$.  (We will give in Section \ref{secwhitney} an argument which involves Segre classes and seems slightly different from Kleiman's and Hironaka's arguments.)

  For the purpose of this paper, let us now introduce further definitions that  will be crucial for the proof  (the notation is a bit heavy, we welcome a better suggestion).
 \begin{Defi} \label{defiflpshch} Let $X$ be smooth. We will denote by ${\rm CH}(X)_{\rm fl_*Ch}$ (for ``flat pushforward of  Chern classes"), resp.  ${\rm CH}(X)_{\rm sm_*Ch}$ (for ``smooth pushforward of   Chern classes") the subgroup generated by cycles of the form $\pi_* z'$ for a flat, resp. smooth, proper morphism $\pi: P\rightarrow X$, with $P$ smooth, and for some cycle $z'\in {\rm CH}(P)_{\rm Ch}$.
\end{Defi}
 Lemma  \ref{propoursegrechern} proved in Section \ref{secdefi} says that  ${\rm CH}^*(X)_{\rm  Ch}\subset {\rm CH}^*(X)$ is generated by smooth pushforwards of classes of complete intersections of divisors. it follows that, in Definition  \ref{defiflpshch}, we could replace  ``cycle in ${\rm CH}(P)_{\rm Ch}$'' by ``cycle in the subring of ${\rm CH}^*(P)$ generated by  divisors'' (see Remark \ref{remaintdiv}). Further easy properties are discussed in Section \ref{secdefi}.

We prove  in Section \ref{secwhitney} the following basic
\begin{prop} \label{proflatpshintro} (Cf. Proposition \ref{proflatpushf})   Cycles  in ${\rm CH}_d(X)_{\rm fl_*Ch} $ are smoothable when  $2d<{\rm dim}\,X$.
\end{prop}
The analogous result for cycles   in ${\rm CH}_d(X)_{\rm sm_*Ch} $ is a standard  statement.
 Proposition \ref{proflatpshintro}  is our motivation to introduce Definition \ref{defiflpshch}. Our second main theorem  is
 the following
\begin{theo}\label{theo0cyclesflpsh}  Let $X$ be a smooth projective variety over a field of characteristic $0$. Then
\begin{eqnarray}\label{eqzeroflpsh}   {\rm CH}(X)=   {\rm CH}(X)_{\rm fl_*Ch}.
\end{eqnarray}
More precisely, any cycle on a smooth projective variety $X$ over a field of characteristic $0$ is obtained as a  pushforward of intersections of divisors on a smooth projective (nonnecessarily connected) variety $Y$ under a flat morphism $f:Y\rightarrow X$.
\end{theo}
\begin{rema}{\rm Equality (\ref{eqzeroflpsh}) combined with Remark \ref{remaintdiv} a priori just says that any cycle on a smooth projective variety $X$ over a field of characteristic $0$ belongs to the subgroup generated by  pushforwards of intersections of divisors on  smooth projective varieties $Y_i$ under flat morphisms $f_i:Y_i\rightarrow X$. However, by allowing disjoint unions (and possibly taking products with $\mathbb{P}^{r_i}$ if we want that the  disjoint union $Y=\sqcup Y_i$ is equidimensional), we find that  pushforwards of intersections of divisors on a smooth projective (nonnecessarily connected) variety $Y$ under a flat morphism $f:Y\rightarrow X$ form a group. Thus the second statement is in fact implied by (\ref{eqzeroflpsh}).}
\end{rema}

Theorem \ref{conjsmoothing} follows from  Theorem \ref{theo0cyclesflpsh} and Proposition \ref{proflatpshintro}.

The  first step in the proof of  Theorem \ref{theo0cyclesflpsh} is  the following Theorem \ref{theopourfilt} proved in Section \ref{sectheoles3prop}, which establishes stability properties of ${\rm CH}(\cdot)_{\rm fl_*Ch}$  under certain non-flat  pushforwards. As there does not seem to exist a standard terminology, we will use here and in the rest of the paper the term {\it complete bundle-section in $X$}, for  ``closed  algebraic subset of codimension $c$ which is the  zero-set  of  a section  of  a vector bundle of rank $c$ on $X$''. We welcome again a better suggestion.

\begin{theo}\label{theopourfilt} (i) (Cf. Proposition \ref{prohypersurafce}.) Given a finite morphism  $j:Y\rightarrow X$, where $Y$ and $X$ are  smooth projective and ${\rm dim}\,Y={\rm dim}\,X-1$,
one has
$$j_*({\rm CH}(Y)_{\rm fl_*Ch})\subset {\rm CH}(X)_{\rm fl_*Ch}.$$

(ii) (Cf. Proposition  \ref{coronewdu1709}.) Let $X$ be a smooth projective variety and let $j:Y\hookrightarrow X$ be the inclusion of a smooth subvariety which is a complete bundle-section in $X$. Then $j_*({\rm CH}(Y)_{\rm fl_*Ch})\subset {\rm CH}(X)_{\rm fl_*Ch}$.

(iii) (Cf. Proposition \ref{problowup}.)  Let $X$ be smooth projective and let $ Y\subset X$ be  a smooth  complete bundle-section in $X$. Let $\tau: \widetilde{X}=B_YX\rightarrow X$ be the blow-up of $X$ along $Y$. Then
$$ \tau_*({\rm CH}(\widetilde{X})_{\rm fl_*Ch})\subset {\rm CH}(X)_{\rm fl_*Ch}.$$
\end{theo}
We will also  show in Section \ref{sectheoles3prop} how Theorem \ref{theopourfilt} easily implies Theorem \ref{theo0cyclesflpsh} for cycles of dimension $\leq 3$ (cf.  Theorem \ref{theo012}).

 Theorem \ref{theo0cyclesflpsh} is then  obtained as a  consequence  of the following ``cbs resolution theorem",  which will be proved in Section \ref{secmainproofofmain}.
\begin{theo} \label{theocbsreso} Let $Z\subset X$ be a smooth subvariety of dimension $d$, with $X$ smooth projective and ${\rm dim}\,X> 4d$. Then, after successive blow-ups of smooth complete bundle-sections $W_i\subset X_i\stackrel{\tau_i}{\rightarrow} X_{i-1}$, $i=0,\ldots,r$, $X_0=X$, the proper transform $Z_r\subset X_r$ is a smooth connected component of a smooth complete bundle-section $Z_r^*$ on ${X}_r$.
\end{theo}

\begin{coro}\label{cibu.14.cor}
  Let $Z\subset X$ be smooth, projective  varieties such that $\dim Z<\frac14 \dim X$.
  Then  after successive blow-ups   $ X_{r+1}\stackrel{\tau_{r+1}}{\rightarrow} X_r\stackrel{\tau_{r}}{\rightarrow}\cdots\to X_0:=X$ along smooth complete bundle-sections $W_i\subset X_i$, there exists  a complete intersection subvariety  $Z_{r+1}\subset X_{r+1}$ such that $\Pi_*(Z_{r+1})=Z$ as effective cycles in $X$, where $\Pi: X_{r+1}\rightarrow X$ is the composition of the $\tau_i$'s.
  \end{coro}
\begin{proof} Let $X_r\to\cdots\to X_0:=X$ be as in Theorem~\ref{theocbsreso}, and
let $ \pi_r:X_{r+1}\to X_r$ be the blow-up of $Z^*_r$.
Let $E_{r+1}$ be the $ \pi_r$-exceptional divisor lying over $Z_r$.
Choose $H_r$  sufficiently ample on $X_r$ such that $|\pi_r^*H_r-E_{r+1}|$
restricts to a very ample divisor on   $E_{r+1}$.
 Then we can choose   general members  $D^i\in |\pi_r^*H_r-E_{r+1}|$ to obtain
 a complete intersection  subvariety
 $Z_{r+1}:=\bigl(E_{r+1}\cap D^1\cap\cdots\cap D^c\bigr)$ satisfying the desired property, where $c=\dim X-\dim Z-1$. \end{proof}
 Theorem \ref{theopourfilt} and Corollary \ref{cibu.14.cor} immediately imply Theorem \ref{theo0cyclesflpsh}.
Indeed, let $Z\subset X$ be a subvariety of dimension $d$. By desingularizing $Z$, we get a smooth subvariety $Z'\subset X\times \mathbb{P}^N$, with $N$ large,  projecting to $Z\subset X$. This way, we are reduced to proving that $[Z]\in{\rm CH}_d(X)_{\rm fl_*Ch}$ when $Z$ is smooth and ${\rm dim}\,X> 4{\rm dim}\,Z$.  We apply Corollary  \ref{cibu.14.cor} to $Z\subset X$, and get $Z_{r+1}\subset {X}_{r+1}\stackrel{\Pi}\rightarrow X$ such that
$[Z]=\Pi_*[Z_{r+1}]$ in ${\rm CH}_d(X)$, where $\Pi$ is a composition of blow-ups along smooth complete bundle-section centers and $Z_{r+1}$ is a complete intersection of divisors in $X_{r+1}$.   By Theorem \ref{theopourfilt}(iii), it follows  that $$\Pi_*[Z_{r+1}]=[Z]\in{\rm CH}_d(X)_{\rm fl_*Ch}.$$

  The methods used to prove Theorems \ref{theo0cyclesflpsh} and  \ref{theopourfilt} do not allow us to prove the stronger result  that $ {\rm CH}_d(X)=  {\rm CH}_d(X)_{\rm sm_*Ch}$. In particular, the proof of the main Proposition \ref{prohypersurafce} (Theorem \ref{theopourfilt}(i) above) does not work if we replace the groups ${\rm CH}_d(Y)_{\rm fl_*Ch}$ and  ${\rm CH}_d(X)_{\rm fl_*Ch}$ respectively  by  ${\rm CH}_d(Y)_{\rm sm_*Ch}$ and  ${\rm CH}_d(X)_{\rm sm_*Ch}$. This leaves open the following
\begin{question}  Are there smooth projective varieties $X$   such that ${\rm CH}(X)\not=  {\rm CH}(X)_{\rm sm_*Ch}$?
\end{question}
As  follows from Theorem \ref{theoab} below, the equality ${\rm CH}(X)=  {\rm CH}(X)_{\rm sm_*Ch}$ is satisfied by abelian varieties but it could be that for the example treated in Lemma \ref{leexample}, or for some homogeneous varieties,  we have ${\rm CH}_0(X)\not={\rm CH}_0(X)_{\rm sm_*Ch}$.

In section \ref{sechomogeneous}, we will give an alternative proof of Theorem \ref{theo0cyclesflpsh} for homogeneous varieties, which does not use the cbs resolution Theorem \ref{theocbsreso}, and  which, in the case of abelian varieties, even gives a stronger result.

\begin{theo}\label{theoab} Let $G$ be an algebraic group and $X$ be a  projective variety which is homogeneous under $G$.

(i) If  there exists a smooth projective $G$-equivariant completion $\overline{G}$ of $G$ which satisfies ${\rm CH}_0(\overline{G})= {\rm CH}_0(\overline{G})_{\rm sm_*Ch}$, then
 \begin{eqnarray}\label{eqnewdupourth2}  {\rm CH }_d(X)_{\rm sm_*Ch}={\rm CH}_d(X)
 \end{eqnarray}
 for all $d$.

(ii)  If  $A$ is  an abelian variety, then ${\rm CH}_d(A)={\rm CH}_d(A)_{\rm sm_*Ch}$ for any $d$.
\end{theo}

 We do not know if the assumption in Theorem \ref{theoab}(i)  is always satisfied. The  stronger version  that there always exists  a smooth projective $G$-equivariant completion $\overline{G}$ of $G$ which satisfies
 \begin{eqnarray}\label{eqnewdupourth2introzerosm} {\rm CH}_0(\overline{G})= {\rm CH}_0(\overline{G})_{\rm Ch}
   \end{eqnarray}
   is wrong for abelian varieties.  For simply connected groups, the   condition (\ref{eqnewdupourth2introzerosm}) seems to be  closely   related to the torsion order of $G$ being $1$ (see \cite{demazure} and Section \ref{secmainproofofmain}), but the precise relation is not obvious to us.

To finish, let us mention that Kleiman also studies in  \cite{kleiman} the strong smoothability problem, which asks whether, given a smooth projective variety $X$ and finitely many smooth subvarieties $W_i\subset X$, any cycle $z\in {\rm CH}(X)$ is a combination of classes of smooth subvarieties $Z_j\subset X$, that have a proper smooth intersection with all the  $W_i$.
Kleiman establishes strong smoothability of cycles with rational coefficients, in the range  $2d-1\leq n$.

Our results do not prove  strong smoothability of cycles with integral coefficients in the range $2d<n$, except in the case of abelian varieties. In general, they imply the weaker statement that, in the situation above, if $2d<n={\rm dim}\,X$,
any cycle $z\in {\rm CH}_d(X)$ is a combination of classes of smooth subvarieties $Z_j\subset X$, that have a proper  intersection with all the  $W_i$  (see Theorem \ref{theopourabstrogn}).

\vspace{0.5cm}

 {\bf Thanks.} {\it CV thanks Olivier Benoist and  Olivier Debarre for  introducing  her to  this subject and for interesting discussions, and  Michel Brion, Laurent Manivel, Nicolas Perrin and Burt Totaro for their help with  homogeneous varieties. Both authors thank the referees for their careful reading and constructive suggestions.}
 \section{Criteria for smoothability \label{secwhitney}}
 We first prove  the following basic  result
 \begin{prop}\label{letranswithflat}  Let $\phi: Y\rightarrow X$ be a  flat morphism between smooth varieties over a field of characteristic $0$. Let $n={\rm dim}\,X$. Then for a smooth subvariety $Z\subset Y$  of dimension $d<\frac{n}{2}$ which is  in general position and such that the restriction  $\phi_{\mid Z}$ is proper, $\phi_{\mid Z}:Z\rightarrow \phi(Z)$ is  an isomorphism, so the closed algebraic subset  $\phi(Z)\subset X$  is smooth. Furthermore, if $n=2d$, $\phi_{\mid Z}$ is an immersion and the image $\phi(Z)$ has finitely many singular points.
 \end{prop}
 Although this will be clear from the proof, let us first make precise what we mean by ``in general position". For the application to the proposition, the general position assumption will be  a transversality condition with respect to $\phi$ and its infinitesimal properties.
  More precisely, let us say that $Z$ is in general position if   $Z$ is the general fiber $\mathcal{Z}_b$ of a family of embeddings \begin{eqnarray}\label{eqmatrice}\begin{matrix} \mathcal{Z}&\stackrel{f}{\rightarrow} &Y
  \\
p\downarrow\,\,\,\,\,\, && &
\\
B&& &
\end{matrix}\end{eqnarray}  where $\mathcal{Z}$ is smooth and $p$ is smooth, which  is very mobile at any point $(x,y)$, $x\not=y$, of $Z\times Z\cong \mathcal{Z}_b\times Z_b$, (that is,  $(f,f):\mathcal{Z}\times_B\mathcal{Z}\rightarrow Y\times Y$ is a submersion at any point $(x,y)$ of $Z\times Z\setminus \Delta_Z$), and whose tangent space is mobile at any point of $Z$, that is, the morphism $$F:\mathbb{P}(T_{\mathcal{Z}/B})\rightarrow \mathbb{P}(T_Y)$$  given by the differential  of the inclusions $f_b:\mathcal{Z}_b\rightarrow Y$, is submersive at any point of $Z$.

The important fact for us is the following
  \begin{rema} \label{rema1}{\rm Assuming $Y\subset \mathbb{P}^N$ is projective of dimension $m$, the general position assumption  will be satisfied by a  general    complete intersection of $m-d$ very ample  divisors. }
 \end{rema}
 For the proof of Proposition  \ref{letranswithflat}, we will use the following consequence of the ``general position'' assumption.
 \begin{lemm}\label{lemtranswhitney}
  (i) If $W\subset Y$ is a closed algebraic subvariety of codimension $>d$ and   $Z\subset Y$ is  a smooth subvariety of dimension $d$  which is in general position, then $Z$ does not intersect $W$.

  (ii)  If
 $W\subset Y\times Y$ is a subvariety of codimension $>2d$, and   $Z\subset Y$ is  a smooth subvariety of dimension $d$  which is in general position, then $Z\times Z$ does not intersect $W$ away from the diagonal of $Z$.

  (iii) If $W\subset \mathbb{P}(T_Y)$ is a subvariety of codimension $\geq 2d$, and   $Z\subset Y$ is  a smooth subvariety of dimension $d$  which is in general position, $\mathbb{P}(T_Z)$ does not intersect $W$.
 \end{lemm}
 \begin{rema}{\rm The general position assumption in each of these statements is relative to the choice of $W$, in the sense that, with the notation of (\ref{eqmatrice}), the Zariski open set of points $b\in B$ for which the fiber $\mathcal{Z}_b$ satisfies the conclusion depends on $W$.}
 \end{rema}
\begin{proof}[Proof of Lemma \ref{lemtranswhitney}]  (i) We use the notations of (\ref{eqmatrice}), with $Z=\mathcal{Z}_b$, $b\in B$ being a general point of $B$. As $f$ is a submersion  along $\mathcal{Z}_b$, there exists  a Zariski neighborhood $U$ of $\mathcal{Z}_b$ in $\mathcal{Z}$ such that any component of $f^{-1}(W)\cap U$  has codimension $>d$ in  $U$. As ${\rm dim}\,B={\rm dim}\,\mathcal{Z}-d$, it follows that
$$p_{\mid f^{-1}(W)\cap U}:  f^{-1}(W)\cap U \rightarrow B$$
cannot be dominant, hence for a general $b\in B$, $\mathcal{Z}_b$ does not intersect $f^{-1}(W)$, that is, $Z=f(\mathcal{Z}_b)$ does not intersect $W$.

(ii) The argument is the same as above with $f$ replaced by $(f,f):\mathcal{Z}\times_B\mathcal{Z}\rightarrow Y\times Y$. The fibers of $(p,p): \mathcal{Z}\times_B\mathcal{Z}\rightarrow B$ are now of dimension $2d$ and $(f,f)$ is a submersion away from the diagonal of $\mathcal{Z}_b$, so if $W\subset Y\times Y$ has codimension $>2d$, $(f,f)^{-1}(W)$ will have codimension $>2d$ in $\mathcal{Z}\times_B\mathcal{Z}\setminus \Delta_{\mathcal{Z}}$, and will not dominate $B$, since ${\rm dim}\,B={\rm dim}\,(\mathcal{Z}\times_B\mathcal{Z})-2d$.

(iii) The argument is the same as above except that  we work now with $F:\mathbb{P}(T_{\mathcal{Z}/B})\rightarrow \mathbb{P}(T_Y)$. We simply observe that the fibers of the natural map $\mathbb{P}(T_{\mathcal{Z}/B})\rightarrow B$ are of dimension $2d-1$.
\end{proof}
 \begin{proof}[Proof of Proposition \ref{letranswithflat}] Let $\Delta_Y\subset Y\times Y$ be the diagonal and let  $Y'\subset Y\times Y\setminus \Delta_Y$ be the closed algebraic subset $Y\times_XY\setminus \Delta_Y$. By flatness of $\phi$, we have ${\rm codim}\,(Y'\subset Y\times Y) =n$. As ${\rm dim}\,Z\times Z=2d<n$ and $Z$ is in general position, $Z\times Z$ does not intersect $Y'$ away from the diagonal by Lemma \ref{lemtranswhitney}(ii), so $\phi_{\mid Z}$ is injective. When $n=2d$, $Z\times Z$  intersects $Y'$ away from the diagonal in at most finitely many points. It remains to prove the infinitesimal statement, for which we only assume that $2d\leq n$.
 Let $Y_k\subset Y$ be the locally closed subset of $Y$ where the rank of $\phi$ is equal to $k$. Then, as we are in characteristic $0$, we have $${\rm dim}\,\phi(Y_k)\leq k,$$ hence by flatness, ${\rm codim}\,(Y_k\subset Y)\geq n-k$, or equivalently ${\rm dim}\,Y_k\leq m+k$, where $m:={\dim}\,Y-n$.
Along $Y_k$, we have the rank $k$ morphism
$$\phi_k:=(\phi_*)_{\mid Y_k}: T_{Y\mid Y_k}\rightarrow (\phi^*T_{X})_{\mid Y_k}$$
with kernel a subbundle  $\mathcal{K}_k \subset T_{Y\mid Y_k}$ of corank $k$. Let $W\subset \mathbb{P}(T_Y)$ be the set of pairs
$(y,u),\,y\in Y,\,u\in {\rm Ker}\,\phi_{*,y}$. The stratification of $Y$ by the $Y_k$'s describes $W$ as a union
$$W=\sqcup_k \mathbb{P}(\mathcal{K}_k).$$
As ${\rm dim}\,Y_k\leq k+m$ and ${\rm rk}\,\mathcal{K}_k=m+n-k$, we get
$${\rm dim}\,\mathbb{P}(\mathcal{K}_k)\leq 2m+n-1$$
for any $k$, and thus ${\rm dim}\,W\leq 2m+n-1$. As ${\rm dim}\,\mathbb{P}(T_Y)=2(m+n)-1$, it follows that
$${\rm codim}\,(W\subset \mathbb{P}(T_Y))\geq n.$$
By Lemma \ref{lemtranswhitney}(iii), $Z$ being of dimension $d$ and in general position with $n\geq 2d$, $\mathbb{P}(T_Z)$ does not intersect $W$. This means that $\phi_{\mid Z}$ is an immersion, which concludes the proof.
\end{proof}
We will combine Proposition \ref{letranswithflat} with  the following easy  result.
\begin{lemm}  \label{propoursegrechern} Let $X$ be smooth of dimension $n$ and let $z\in{\rm CH}_d(X)$ be a cycle. Assume that $z$ belongs to the subring ${\rm CH}^*(X)_{\rm Ch}$ of ${\rm CH}^*(X)$ generated by  Chern classes $c_i(E)$ for any coherent sheaf $E$ on $X$. Then there exist a smooth  variety $Y$  and a smooth projective morphism $f:Y\rightarrow X$ such that $z=f_*z'$ in ${\rm CH}(X)$, where $z'\in{\rm CH}(Y)$ belongs to the subring generated by divisors on $Y$.
\end{lemm}
\begin{proof} First of all, using finite locally free resolutions and the Whitney formula, we know that $z$ belongs to the subring of ${\rm CH}^*(X)$ generated by the Chern classes $c_i(E)$ for any locally free coherent sheaf $E$ on $X$. Secondly, we can replace in this statement the Chern classes by the Segre classes, since the total Segre and Chern classes $s(E)$ and $c(E)$ satisfy   the relation $$s(E)=c(E)^{-1},\,\, c(E)=s(E)^{-1},$$ so any polynomial with integral coefficients in the Segre classes is a  polynomial with integral coefficients in the Chern  classes and vice-versa.

It thus suffices to prove that any monomial $s_{i_1}(E_1)\ldots s_{i_k}(E_k)\in {\rm CH}(X)$, where the $E_i$'s are locally free sheaves on $X$ of rank $r_i$, satisfies the conclusion of  Lemma \ref{propoursegrechern}. This statement follows  from the definition of Segre classes (see \cite{fulton}). Indeed, let $\pi_i: \mathbb{P}(E_i)\rightarrow X$ be the projectivization of $E_i$ and let $H_i\in {\rm Pic}(\mathbb{P}(E_i))$ with first Chern class $c_1(H_i)\in{\rm CH}^1(\mathbb{P}(E_i))$ be the dual of its Hopf line bundle (so that $R^0\pi_{i*}H_i=E_i^*$). Then
\begin{eqnarray} \label{eqpoursegre} s_j(E_i)=\pi_{i*} (c_1(H_i)^{j+r_i-1})\,\,{\rm in}\,\, {\rm CH}(X).\end{eqnarray}
It follows from (\ref{eqpoursegre}) and the projection formula that
\begin{eqnarray} \label{eqpoursegre2}s_{i_1}(E_1)\ldots s_{i_k}(E_k)=\pi_*({\rm pr}_1^*c_1(H_1)^{i_1+r_1-1}\ldots {\rm pr}_k^*c_1(H_k)^{i_k+r_k-1})\,\,{\rm in}\,\, {\rm CH}(X),\end{eqnarray}
where $\pi:\mathbb{P}(E_1)\times_X\ldots\times_X\mathbb{P}(E_k)\rightarrow X$ is the fibred product of the $\pi_i: \mathbb{P}(E_i)\rightarrow X$ and ${\rm pr}_i$ is the projection from $\mathbb{P}(E_1)\times_X\ldots\times_X\mathbb{P}(E_k)$ to its $i$-th factor.
\end{proof}
\begin{coro}   \label{coropropoursegrechern} Let $X$ be smooth of dimension $n$ and let $z\in{\rm CH}_d(X)_{\rm Ch}$, with $2d<n$. Then
 $z$ is smoothable, that is,  $z$ is  rationally equivalent to a  cycle $Z'=\sum_i n_i Z'_i$, where $Z'_i\subset X$ is smooth.
\end{coro}
\begin{proof}
Using Lemma \ref{propoursegrechern}, the result follows from Proposition \ref{letranswithflat} by Remark \ref{rema1}.
\end{proof}
\begin{coro} (Hironaka \cite{hironaka}, Kleiman \cite{kleiman}) If $X$ is smooth and $2d<{\rm dim}\,X$, any cycle $z\in {\rm CH}_d(X)_\mathbb{Q}$ is rationally equivalent to a smooth cycle with $\mathbb{Q}$-coefficients. More precisely $(c-1)! z$ is smoothable, where $c:=n-d$ is the codimension of $z$.
\end{coro}
\begin{proof} Indeed, it suffices to prove the result when $z=[Z]$ is the class of a subvariety $Z$ of $X$ of dimension $d$. Let $\mathcal{O}_Z$ be the structural sheaf of $Z$, seen as a coherent sheaf on $X$.  It follows from the Grothendieck-Riemann-Roch formula (see   \cite[Example 15.3.1]{fulton})  that
\begin{eqnarray}\label{eqformulap1fact} c_{n-d}(\mathcal{O}_Z)=(-1)^{n-d-1}(n-d-1)![Z]\,\,\in\,\,{\rm CH}_d(X)_{\rm Ch}\subset {\rm CH}_d(X),
\end{eqnarray} so Corollary  \ref{coropropoursegrechern} applies.
\end{proof}
\begin{rema}{\rm In \cite{hironaka}, which does not use Segre classes but the splitting principle to reduce Chern classes to products of divisors, the coefficient $(c-1)!$ appears multiplied by a constant, which is possibly $1$.}
\end{rema}
\begin{rema}{\rm Kleiman in \cite{kleiman} argues differently by studying singularities of Schubert varieties and    proves a result which is of a different nature, as it also includes the cases where $n=2d$ or $2d-1$, which are above the middle dimension.}
\end{rema}
Combining  Lemma  \ref{propoursegrechern} and Proposition  \ref{letranswithflat}, we get the following criterion
\begin{prop}\label{proflatpushf}  Let $\phi: Y\rightarrow X$ be a proper flat morphism between smooth varieties. Then for any cycle $z\in{\rm CH}_d(Y)_{\rm Ch}$ with $2d<n={\rm dim}\,X$, the class $z'=\phi_*z\in{\rm CH}_d(X)$ is smoothable on $X$.
\end{prop}
\begin{proof}  By Lemma \ref{propoursegrechern}, the cycles $z$ as above are of the form
$\pi_* z'$ for a smooth proper morphism $\pi: P\rightarrow Y$ and for some cycle $z'\in {\rm CH}(P)$ which is a combination with integral coefficients  of  intersections of divisors on $P$. By Remark \ref{rema1}, Proposition \ref{letranswithflat} applies to $z'$ and the flat morphism $\phi\circ \pi:P\rightarrow X$, proving the statement.
\end{proof}

\section{Flat pushforwards of Chern classes  \label{secproofflpsh}}
\subsection{Comments on Definition \ref{defiflpshch}\label{secdefi}}
We start with the following remarks on Definition \ref{defiflpshch}.
\begin{rema}\label{remaintdiv} {\rm   By Lemma \ref{propoursegrechern}, in the definition of
${\rm CH}(X)_{\rm fl_*Ch}$ and  ${\rm CH}(X)_{\rm sm_*Ch}$,  we can replace ``elements of ${\rm CH}(P)_{\rm Ch}$'' by ``intersections of divisors on $P$''. Indeed, if $Y$ is smooth, $p:Y\rightarrow X$ is a proper  flat morphism, and $z\in {\rm CH}_d(Y)_{\rm Ch}$, there exist
by Lemma \ref{propoursegrechern} a smooth variety $P$ and a smooth morphism $p':P\rightarrow Y$ such that
$$z=p'_*(w)\,\,{\rm in}\,\, {\rm CH}_d(Y),$$ where $w$ belongs to the subring of ${\rm CH}(P)$ generated by divisors. The morphism $p\circ p': P\rightarrow X$ is flat and projective, and $p_*z=(p\circ p')_*w$.}
\end{rema}
\begin{rema} \label{remastable} If $f: Y\rightarrow X$ is a flat (resp. smooth)  morphism between smooth  projective varieties,
one has
$f_*({\rm CH}(Y)_{\rm fl_*Ch})\subset {\rm CH}(X)_{\rm fl_*Ch}$, resp. $f_*({\rm CH}(Y)_{\rm sm_*Ch})\subset {\rm CH}(X)_{\rm sm_*Ch}$.
\end{rema}
Let us now establish a few elementary facts.
\begin{lemm} \label{remaalgrat} (i)   One has ${\rm CH}(X)_{\rm Ch}\subset {\rm CH}(X)_{\rm sm_*Ch}\subset {\rm CH}(X)_{\rm fl_*Ch} $.

(ii) The subgroup  ${\rm CH}(X)_{\rm sm_*Ch}$ is a subring of ${\rm CH}(X)$.

(iii) The subgroup  ${\rm CH}(X)_{\rm fl_*Ch}$ is a module over  the ring ${\rm CH}(X)_{\rm sm_*Ch}$.

\end{lemm}

\begin{proof} (i)  The second inclusion is obvious since smoothness implies flatness.

(ii) and (iii)  Let $p_1: P_1\rightarrow X$, $p_2:  P_2\rightarrow X$ be proper morphisms with $P_1,\,P_2,\,X$ smooth and assume $p_1$ is smooth, $p_2$ is flat. Then $P_{12}:=P_1\times_XP_2$ is smooth. If $Z_1$, resp. $Z_2$ are intersections of divisors on $P_1$, resp. $P_2$,
their pull-backs $Z'_1$, resp. $Z'_2$ to  $P_{12}$ via the  projections
$$p'_1: P_{12}\rightarrow P_2,\,\,p'_2: P_{12}\rightarrow P_1$$
are also  intersections of divisors, and the projection formula
gives
$$ p_{1*}Z_1\cdot p_{2*}Z_2=p_{12*}(Z'_1\cdot Z'_2)\,\,{\rm in}\,\,{\rm CH}(X),$$
where $p_{12}: P_1\times_XP_2\rightarrow X$  is the natural morphism.
This proves (ii) and (iii) since  $p_{12}$ is  flat and it is  smooth if $p_2$ is smooth.
\end{proof}

Another useful lemma is the following
\begin{lemm} \label{lebasicpi*} Let $\phi: Y_1\rightarrow Y_2$ be a  morphism, with $Y_1,\,Y_2$ smooth. Then

(i)  One has
\begin{eqnarray} \label{eqincldu1709stable1} \phi^*({\rm CH}(Y_2)_{\rm sm_*Ch})\subset {\rm CH}(Y_1)_{\rm sm_*Ch},\end{eqnarray}

(ii) If $\phi$ is smooth, then
\begin{eqnarray} \label{eqincldu1709stable2}\phi^*({\rm CH}(Y_2)_{\rm fl_*Ch})\subset {\rm CH}(Y_1)_{\rm fl_*Ch}.\end{eqnarray}
\end{lemm}
\begin{proof} Let $\psi: W\rightarrow Y_2$ be a flat (resp. smooth) projective morphism. Then
$$\psi_1: W_1:= W\times_{Y_2} Y_1\rightarrow Y_1$$
is flat (resp. smooth). Furthermore, if either $\psi$ is smooth (Case (i))  or  $\phi$ is smooth (Case (ii)),  $W_1$ is smooth.

Let $\phi_1: W_1\rightarrow W$ be the first projection. If  $\gamma\in {\rm CH}(W)_{\rm Ch}$, we have
$\phi_1^*\gamma\in {\rm CH}(W_1)_{\rm Ch}$, and furthermore we have by \cite[Proposition 1.7]{fulton}
$$ \psi_{1*}(\phi_1^*\gamma)=\phi^*(\psi_*\gamma)\,\,{\rm in}\,\,{\rm CH}(Y_1).$$
This proves (\ref{eqincldu1709stable1}) and (\ref{eqincldu1709stable2}).
\end{proof}

\subsection{Examples of cycles not in ${\rm CH}(X)_{\rm Ch}$ \label{secexamples}}
Theorem \ref{theo0cyclesflpsh}  is interesting when the considered cycles do not belong to ${\rm CH}(X)_{\rm Ch}$. Besides the  case of $0$-cycles on very general abelian  varieties with high degree polarization (see \cite{debarre}),  some   hypersurfaces in projective space provide  such examples. For example, we have
\begin{lemm}\label{leexample} Let  $X\subset \mathbb{P}^4$  be a very general hypersurface of degree $64$. Then the class of a point  $x\in X$  does not belong to ${\rm CH}_0(X)_{\rm Ch}$.   More precisely, for  any vector bundle $E$ on $X$, the degree ${\rm deg}\,c_3(E)$ is divisible by $2$.
\end{lemm}
\begin{proof}     Let $E$ be a vector bundle of rank $r$  on $X$.  By the Hirzebruch-Riemann-Roch formula, the holomorphic Euler-Poincar\'e characteristic of $E$ is  given by the formula
\begin{eqnarray}\label{eqchihol}  \chi(X,E)= \int_X {\rm ch}(E){\rm td}(X)=\alpha c_3(E)+\beta c_2(E)c_1(E)+\gamma c_2(E) c_1(X)+  q(E),\end{eqnarray}
where the constants $\alpha,\,\beta,\,\gamma$, which are independent of $E$, are rational and the quantity $q(E)$ is the part of the Riemann-Roch polynomial (in the Chern classes of $E$) which involves only  $c_1(E)$ and the rank of $E$, and is  an integer since it is equal to
$$\int_X(r-1){\rm td}_3(X)+{\rm ch}({\rm det}\,E){\rm td}(X)=(r-1)\chi(X,\mathcal{O}_X)+\chi(X,{\rm det}\,E).$$
The constants $\alpha$ and $\beta$ are obtained by expressing ${\rm ch}_3(E)$ as a polynomial in the Chern classes $c_i(E)$. One gets
\begin{eqnarray} \label{eqalphabeta} \alpha=\tfrac12,\,\beta= -\tfrac12.\end{eqnarray}
Finally, the constant $\gamma$ is obtained by expressing ${\rm ch}_2(E){\rm td}_1(X)$ using the Chern classes of $E$. One gets
\begin{eqnarray} \label{eqgamma} \gamma=-\tfrac12.\end{eqnarray}

It is proved in \cite{kollar} that for $X$ as above, any curve $C\subset X$ has degree divisible by $2$.  It follows that the numbers $\int_Xc_2(E)c_1(E)$ and $\int_Xc_2(E) c_1(X)$ are even. We thus deduce from
(\ref{eqalphabeta}) and  (\ref{eqgamma}) that
$\chi(X,E)=\frac{1}{2}\int_Xc_3(E)+ s$, where $s$ is an integer. Thus the degree of $c_3(E)$ has to be an even integer.
\end{proof}
\begin{rema}{\rm We see from the proof above that the obstruction to the existence of a vector bundle $E$ on $X$ (or more generally an element of $K_0(X)$) with ${\rm deg}\,c_3(E)=1$ comes from the defect of the integral Hodge conjecture for degree $4$ Hodge classes on $X$. Conversely, if the integral Hodge conjecture for degree $4$ Hodge classes on $X$ holds true, then, denoting by  $H\in{\rm CH}^1(X)$  the class of a hyperplane section, the generator  $a\in H^4(X,\mathbb{Z})$ such that $\langle a,[H]\rangle=1$  is algebraic, that is, $a=[Z]$ for some $1$-cycle $Z\in {\rm CH}_1(X)$. As we have
$$ {\rm CH}_1(X)={\rm CH}^2(X)={\rm CH}^2(X)_{\rm Ch}$$
by formula (\ref{eqformulap1fact}),
the cycle $Z$ belongs to ${\rm CH}^2(X)_{\rm Ch}$, so the cycle $H\cdot Z$ belongs to ${\rm CH}^3(X)_{\rm Ch}$. Hence there exists a degree $1$ element in
${\rm CH}^3(X)_{\rm Ch}$ in this case.}
\end{rema}
 Other examples of smooth projective varieties $X$ for which ${\rm CH}(X)\not={\rm CH}(X)_{\rm Ch}$ are given by generalized flag manifolds for certain affine algebraic groups with torsion index $>1$
 (see \cite{demazure}  and \cite{totaro}, \cite{totaro2} where this notion is discussed and computed for many groups).   Merkurjev proved in  \cite{merkurjev} that for a simply connected semisimple algebraic group $G $, and for a closed subgroup $H$, the $K_0$-ring of $G/H$  is generated by classes of homogeneous vector bundles on $G/H$ that come from representations of $H$. If furthermore $H$ is a Borel subgroup of $G$, then homogeneous vector bundles on $G/H$ coming from representations of $H$  are direct sums of line bundles. In the last case, it follows that the subgroup
 $${\rm CH}_0(G/H)_{\rm Ch}\subset {\rm CH}_0(G/H)$$
 is also the subgroup generated by products of divisor classes. By definition of the torsion index of $G$, the index of the latter subgroup is a multiple of  the torsion index of $G$. The computations in \cite{totaro}, \cite{totaro2} thus give plenty of examples where ${\rm CH}_0(G/H)_{\rm Ch}\subset {\rm CH}_0(G/H)$ is a proper subgroup.

\subsection{ \label{sectheoles3prop} Some stability results for ${\rm CH}(X)_{\rm fl_*Ch}$  }

We will give in this section  the  proof of Theorem \ref{theopourfilt}.
 It  will rely on the following three propositions.
\begin{prop}\label{prohypersurafce} Let $X,\,Y$ be  smooth projective varieties with  ${\rm dim}\,Y={\rm dim}\,X-1$, and let $j:Y\rightarrow X$ be a  finite morphism.  Then
$$j_*({\rm CH}(Y)_{\rm fl_*Ch})\subset {\rm CH}(X)_{\rm fl_*Ch}.$$
\end{prop}

\begin{proof}
Let $T=B_{\Gamma_j}(Y\times X)$ be the smooth projective variety obtained by blowing-up the graph $\Gamma_j$ of $j$ in $Y\times X$. Let
$$\tau: T\rightarrow Y\times X$$
be the blow-up map and let  ${\rm pr}_Y,\,{\rm pr}_{X}$ be the two projections from $Y\times X$ to $Y$ and $X$. We denote
$$p:= {\rm pr}_Y\circ \tau:  T\rightarrow Y,\,\,q:=  {\rm pr}_{X}\circ \tau: T\rightarrow X$$
 the two natural morphisms.
\begin{lemm} \label{lepsmooth} The morphism  $p$ is smooth and the morphism  $q$ is flat.
\end{lemm}
\begin{proof} Indeed, the fiber of $p$ over $y\in Y$ is the blow-up of $X$ along $j(y)$, which is smooth. The fiber of $q$ over  $x\in X$ is  isomorphic to $Y$ when $x\not\in j(Y)$, hence it has dimension $n-1$, $n={\rm dim}\,X$. We claim that all the  fibers of $q$ have dimension $\leq n-1$.  To see this, we observe that
$$q^{-1}(x)=\tau^{-1}(Y\times \{x\})$$
is the set-theoretic union of several  components, some  being contained in the exceptional divisor $E$ over $\Gamma_j$ and mapping via $\tau$ to $(Y\times \{x\})\cap  \Gamma_j$, the other being birational to $Y$. The component which is birational to $Y$  has dimension $ n-1$.  The other components are also of dimension $\leq n-1$, since the morphism $\tau_{\mid E}: E\rightarrow \Gamma_j\cong Y$ has fibers of dimension $n-1$, and $$(Y\times \{x\})\cap  \Gamma_j\cong j^{-1}(x)\subset Y\cong \Gamma_j$$ has dimension $0$ because $j$ is finite.
This proves the claim. The fibers are  thus equidimensional, hence $q$ is flat since both $T$ and $X$ are smooth.
\end{proof}
For any  class $w \in {\rm CH}_d(Y)_{\rm fl_*Ch}$,  there exist by definition  a (not necessarily connected but equidimensional)  smooth projective variety $W$, a flat morphism $\phi: W\rightarrow Y$,  and divisors $D_1,\ldots,\,D_{N-d}\in {\rm CH}^1(W)$, $N:={\rm dim}\,W$, such that
\begin{eqnarray}\label{eqpourYzero} w=\phi_*(D_1\cdot\ldots \cdot D_{N-d})\,\,{\rm in}\,\,{\rm CH}_d(Y).
\end{eqnarray}
 We now observe that
\begin{eqnarray}\label{eqpourYzero1} j_*w= \Gamma_{j*}(w)= {\rm pr}_{X*}({\rm pr}_Y^*w\cdot \Gamma_j) \end{eqnarray} in ${\rm CH}_d(X)$. Furthermore, we have as usual
\begin{eqnarray}\label{eqpourYzero2}  \Gamma_j=\pm \tau_* E^{n} \,\,{\rm in}\,\,{\rm CH}^n(Y\times X),\, n={\rm dim}\,X.\end{eqnarray}
Let now  $$ W_{T}:= W\times_Y T$$
with first projection $p_W$ to $W$, second projection $p_T$ to $T$ and
morphism $$\psi:= q\circ p_T:  W_{T}\rightarrow X.$$
We first observe that $W_{T}$ is smooth by Lemma \ref{lepsmooth}. Next,
 combining (\ref{eqpourYzero}),  (\ref{eqpourYzero1}), and  (\ref{eqpourYzero2}), and applying the projection formula, we get
\begin{eqnarray}\label{eqpourYzerofinal} j_* w=\pm \psi_* (p_W^*(D_1\ldots D_{N-d})\cdot p_T^* E^n)\,\,{\rm in}\,\,{\rm CH}_d(X),
\end{eqnarray}
which proves that $j_*w$ belongs to ${\rm CH}_d(X)_{\rm fl_*Ch}$, since  the  morphism  $\psi$ is flat, being the composition of the two flat  morphisms $q$ and $p_T$. The proof of Proposition \ref{prohypersurafce} is finished.
\end{proof}
 Proposition \ref{prohypersurafce} has the following consequences.
\begin{prop} \label{coronewdu1709} Let $X$ be smooth projective and let $j: Y\hookrightarrow X$ be the inclusion of a smooth projective subvariety which is the zero-set of a transverse section $\sigma$ of a vector bundle $E$ on $X$. Then
\begin{eqnarray} \label{eqj*Y} j_*({\rm CH}(Y)_{\rm fl_*Ch})\subset {\rm CH}(X)_{\rm fl_*Ch}.
\end{eqnarray}
\end{prop}
\begin{proof} We prove the result by induction on the rank  of $E$, the case of rank $1$ being a particular case of Proposition \ref{prohypersurafce}. Let $E$ be a rank $r$ vector bundle on $X$ and let $\pi: \mathbb{P}(E^*)={\rm Proj}\,({\rm Sym}^*E)\rightarrow X$ be the projectivization of $E^*$. Let $\pi^*E\rightarrow \mathcal{H}$ be the quotient line bundle on $\mathbb{P}(E^*)$. The section $\pi^*\sigma\in H^0(\mathbb{P}(E^*),\pi^*E)$ projects to  a section $\sigma'$ of $\mathcal{H}$ and we have
\begin{lemm}\label{lepourpurevedu1709}  (i) The zero-locus of $\sigma'$ is a smooth hypersurface $X'$ of $\mathbb{P}(E^*)$.

(ii) Furthermore, the induced section $\sigma''$ of $F:={\rm Ker}\,(\pi^*E\rightarrow \mathcal{H})$ on $X'$ is transverse with zero-locus $\pi^{-1}(Y)=\mathbb{P}(E^*_{\mid Y})\subset X'\subset \mathbb{P}(E^*)$.\end{lemm}
\begin{proof} (i) The vanishing locus of $\sigma'$ is a $\mathbb{P}^{r-2}$-bundle over the open subset $X\setminus Y$ of $X$ where $\sigma\not=0$, hence it is smooth over $X\setminus Y$. It obviously contains $\pi^{-1}(Y)$ and it remains to show that it is smooth there, which is easy.

(ii) The vanishing locus of $\sigma''$ on $X'$ equals scheme-theoretically the vanishing locus of the section $\pi^*\sigma$ of $\pi^*E$ on $\mathbb{P}(E^*)$, hence equals $\pi^{-1}(Y)$. It is thus smooth of codimension $r-1$ in $X'$.
\end{proof}
Denoting by $\pi_Y:\mathbb{P}(E^*_{\mid Y})\rightarrow Y$ the restriction of $\pi$ over $Y$,  we know by  Lemma \ref{lebasicpi*} that $\pi^*_Y:{\rm CH}(Y) \rightarrow {\rm CH}(\mathbb{P}(E^*_{\mid Y}))$ maps
${\rm CH}(Y)_{\rm fl_*Ch}$ to ${\rm CH}(\mathbb{P}(E^*_{\mid Y}))_{\rm fl_*Ch}$. Denoting by $$j':\mathbb{P}(E^*_{\mid Y})\hookrightarrow X',\,\,j'': X'\hookrightarrow \mathbb{P}(E^*)$$ the inclusion maps,  we get, first  by the induction hypothesis on the rank and Lemma \ref{lepourpurevedu1709}, and secondly  by Proposition \ref{prohypersurafce}, that
$$j'_*({\rm CH}(\mathbb{P}(E^*_{\mid Y}))_{\rm fl_*Ch})\subset {\rm CH}(X')_{\rm fl_*Ch} ,\,\, j''_*({\rm CH}(X')_{\rm fl_*Ch})\subset {\rm CH}(\mathbb{P}(E^*))_{\rm fl_*Ch} .
$$
We conclude that the map $\gamma:=j''_*\circ j'_*\circ \pi_Y^*:{\rm CH}(Y)\rightarrow {\rm CH}(\mathbb{P}(E^*))$ has the property that
$$\gamma({\rm CH}(Y)_{\rm fl_*Ch})\subset {\rm CH}(\mathbb{P}(E^*))_{\rm fl_*Ch}.$$
Recalling from \cite[Proposition 1.7]{fulton} that $\gamma=\pi^*\circ j_*$, we thus proved that
$$\pi^*\circ j_*({\rm CH}(Y)_{\rm fl_*Ch})\subset {\rm CH}(\mathbb{P}(E^*))_{\rm fl_*Ch}.$$
Let $h=c_1(\mathcal{H})\in{\rm CH}^1(\mathbb{P}(E^*))$. By Lemma \ref{remaalgrat}, we have
$$h^{r-1}{\rm CH}(\mathbb{P}(E^*))_{\rm fl_*Ch}\subset {\rm CH}(\mathbb{P}(E^*))_{\rm fl_*Ch}$$
and by Remark \ref{remastable}, $\pi_* ({\rm CH}(\mathbb{P}(E^*))_{\rm fl_*Ch})\subset {\rm CH}(X)_{\rm fl_*Ch}$.
As $\pi_*(h^{r-1}\pi^*z)=z$ for any $z\in {\rm CH}(X)$, we conclude that for any $z\in {\rm CH}(Y)_{\rm fl_*Ch}$,
$$j_*z=\pi_*(h^{r-1}\pi^*(j_*z))\in {\rm CH}(X)_{\rm fl_*Ch}.$$
\end{proof}
Another consequence of Proposition \ref{prohypersurafce} is the following
\begin{prop} \label{problowup}  Let $X$ be smooth projective and let $ Y\subset X$ be a smooth projective subvariety which is the zero-set of a transverse section $\sigma$ of a vector bundle $E$ on $X$. Let $\tau: \widetilde{X}=B_YX\rightarrow X$ be the blow-up of $X$ along $Y$. Then
\begin{eqnarray} \label{eqtau*Y} \tau_*({\rm CH}(\widetilde{X})_{\rm fl_*Ch})\subset {\rm CH}(X)_{\rm fl_*Ch}.
\end{eqnarray}
\end{prop}
\begin{proof} Let $\pi: \mathbb{P}(E)={\rm Proj}({\rm Sym}^*E^*)\rightarrow X$ be the projectivization of $E$. Then the section $\sigma$ gives a rational section $X\dashrightarrow \mathbb{P}(E)$ of $\pi$, whose image is isomorphic to $\widetilde{X}$. Furthermore, as  a local computation shows, $\widetilde{X}\subset \mathbb{P}(E)$ is the zero-set of a transverse section $\overline{\sigma}$ of the quotient vector bundle $F:=\pi^* E/\mathcal{S}$ on $\mathbb{P}(E)$, namely, $\overline{\sigma}$ is the projection of $\pi^*\sigma$ in $\pi^* E/\mathcal{S}$, where $\mathcal{S}\subset \pi^*E$ is the tautological subbundle. We have
\begin{eqnarray} \label{eqtau*Ydanspre}\tau_*=\pi_*\circ j_*: {\rm CH}(\widetilde{X})\rightarrow {\rm CH}(X),
\end{eqnarray}
where $j:\widetilde{X}\rightarrow \mathbb{P}(E)$ is the inclusion map. By Proposition \ref{coronewdu1709}, we have
$$j_*({\rm CH}(\widetilde{X})_{\rm fl_*Ch})\subset {\rm CH}(\mathbb{P}(E))_{\rm fl_*Ch}$$
and  $\pi_*({\rm CH}(\mathbb{P}(E))_{\rm fl_*Ch})\subset {\rm CH}(X)_{\rm fl_*Ch}$ by Remark \ref{remastable}. Hence (\ref{eqtau*Ydanspre}) implies (\ref{eqtau*Y}).
\end{proof}

As a consequence of these propositions, we give  the easy  proof of  Theorem \ref{theo0cyclesflpsh} for cycles of dimension at most 3, and more generally of the following result.
 \begin{theo} \label{theo012} For any smooth projective variety $X$, we have
\begin{eqnarray}\label{eqinclusionweak} (d-2)!{\rm CH}_d(X)\subset {\rm CH}_d(X)_{\rm fl_*Ch},
\end{eqnarray}
with the convention that $(d-2)!=1$ if $d\leq 2$. In particular, for $d\leq 3$, we have ${\rm CH}_d(X)= {\rm CH}_d(X)_{\rm fl_*Ch}$.
\end{theo}
\begin{proof}  Let $X$ be smooth projective of dimension $n$ and let
$Z\subset X$ be a subvariety of dimension $d$. We choose a desingularization
$\widetilde{Z}\rightarrow Z$ of $Z$ and an embedding $\widetilde{Z}\subset X\times \mathbb{P}^m$ over $X$ for some $m$. As the projection $p_X: X\times \mathbb{P}^m \rightarrow X$ to $X$ is flat, it suffices to prove that
\begin{eqnarray} \label{eqd-1factoriel} (d-2)! \widetilde{Z}\in {\rm CH}( X\times \mathbb{P}^m)_{\rm fl_*Ch},
\end{eqnarray}
as it implies by Remark \ref{remastable} that $ (d-2)! {Z}\in {\rm CH}( X)_{\rm fl_*Ch}$, which is the contents of (\ref{eqinclusionweak}). In other words, letting $$Z'=\widetilde{Z},\,X'=X\times \mathbb{P}^m,$$
we reduced to the case of the class of a smooth subvariety $Z'\subset X'$,   which we treat now.  If ${\rm dim}\,X'\leq 2d-1$, this is finished by formula (\ref{eqformulap1fact}) since then ${\rm codim}\,(Z'\subset X')\leq d-1$. If not,  let  $Y$ be a smooth general complete intersection of sufficiently ample hypersurfaces  in $X'$  containing  $Z'$, with $Y$ of dimension $2d$. Such $Y$ exists by   Lemma \ref{letroptropfacile} below,  since $Z'$ is smooth of dimension $d$ and ${\rm dim}\,X'\geq 2d$. Let $j: Y\hookrightarrow  X'$ be the inclusion of  $Y$.  As $Y$ is a smooth complete bundle-section in $X'$, we have by Proposition \ref{coronewdu1709}
$$j_*({\rm CH}(Y)_{\rm fl_*Ch})\subset  {\rm CH}(X')_{\rm fl_*Ch}.$$

In order to prove Theorem \ref{theo012},  it thus suffices  to prove that the class $z'$ of $Z'$ in $Y$ satisfies \begin{eqnarray}\label{eqpourYzerofinalameliored-2}(d-2)!z'\in {\rm CH}_d(Y)_{\rm fl_*Ch}.
\end{eqnarray}
Let $Y'\subset Y$ be  a general sufficiently ample hypersurface containing $Z'$.
\begin{lemm}\label{letroptropfacile} Let $N$ be a smooth variety  and  $ M\subset N$ be a smooth subvariety of dimension $d$ and codimension $c$.  Then a general sufficiently ample hypersurface $H\subset N$ containing $M$ has  ordinary quadratic singularities along  a smooth subvariety $D\subset M$ of dimension $d-c$. In particular $H$ is smooth if $d<c$ and,  if $d=c$, $H$ has isolated ordinary quadratic singularities.
\end{lemm}
\begin{proof}  (See also \cite[Theorem 2.1 and Corollary 2.5]{diazharbater}.) By Bertini, the singularities of $H$ are on $M$, and they correspond to the zeroes of the differential
\begin{eqnarray}
\label{eqsectionetsadiff} d\sigma\in H^0(M,N_{M/N}^*(H))
\end{eqnarray}
along $M$ of the defining equation $\sigma$ of $H$. When $H$ is sufficiently ample, the section (\ref{eqsectionetsadiff}) is a  general section of the bundle $N_{M/N}^*(H)$ and this bundle  is globally  generated, so the zero-locus of $d\sigma$  is transverse, hence the singular locus of $Y'$ is smooth of dimension $d-c$ (and empty if $d-c<0$). In fact, the transversality of the  section  $d\sigma$ also implies that the singularities are ordinary quadratic as a local computation shows.
\end{proof}
By Lemma \ref{letroptropfacile}, the hypersurface $Y'$ above  is desingularized by a single blow-up along the finite  set $W\subset Y'$ of  its  singular points. We choose now a general $0$-dimensional complete intersection $W'\subset Y$ containing $W$. We have $W'=W\cup W''$, where the set $W''$ is disjoint from $Y'$. It follows that the blow-up $\widetilde{Y}$ of $Y$ along $W'$ contains the blow-up $\widetilde{Y'}$ of $Y'$ along $W$ as a smooth hypersurface.
Let $\widetilde{Z'}\subset \widetilde{Y'}$ be the proper transform of $Z'$ and denote by $\tilde{z}'$ its class in ${\rm CH}_d(\widetilde{Y'})$. As the codimension of $\widetilde{Z'}$ in $\widetilde{Y'}$ is $d-1$, we have $(d-2)!\tilde{z}'\in{\rm CH}_d(\widetilde{Y'})_{\rm Ch}$ by (\ref{eqformulap1fact}). We now apply Proposition \ref{prohypersurafce} to the inclusion $i$ of
$\widetilde{Y'}$ in $ \widetilde{Y}$ and conclude that $$(d-2)!i_*\tilde{z}'\in {\rm CH}_d(\widetilde{Y})_{\rm fl_*Ch}.$$  As the morphism $\tau: \widetilde{Y}\rightarrow Y$ blows-up the smooth complete bundle-section $W'$ in $Y$, we finally get
$$(d-2)!z'=\tau_*((d-2)!i_*\tilde{z}')\in {\rm CH}_d(Y)_{\rm fl_*Ch}$$
by Proposition \ref{problowup}.
\end{proof}

 \section{The cbs resolution  Theorem\label{secmainproofofmain}}

The main result of this section  is Theorem~\ref{cibu.14.thm}, which says that
a  smooth  subvariety of a  smooth  variety becomes an irreducible  component of a smooth
complete bundle-section after a suitable sequence of blow-ups,  whose centers are also
smooth complete bundle-sections.

In this section we work over an infinite perfect field. All
 varieties are allowed to be reducible, but assumed pure dimensional.

  {\bf Blow-up sequences.}
  A  {\it blow-up sequence}  is a sequence of morphisms
  \begin{eqnarray}\label{bus.say1}
  Y_r\stackrel{\pi_{r-1}}{\longrightarrow} Y_{r-1}
  \stackrel{\pi_{r-2}}{\longrightarrow}\cdots
  \stackrel{\pi_{0}}{\longrightarrow} Y_0,
  \end{eqnarray}
  where each $\pi_i:Y_{i+1}\to Y_i$ is the blow up of a subscheme
  $C_i\subset Y_i$, called the {\it center } of the blow-up.

  Let $W_0\subset Y_0$ be a subscheme. If the images of the centers $C_i$ are nowhere dense in $W_0$, then we let $W_i\subset Y_i$ denote the {\it  birational  transform} of $W_0$  (also called {\it proper transform} of $W_0$).

  Here we only deal with blow-up sequences where $Y_0$ is smooth, and the $C_i$ are smooth and pure dimensional. In this case all the $Y_i$ are smooth.

  We say that a blow-up sequence as in (\ref{bus.say1}) is a {\it complete bundle-section blow-up sequence} (abbreviated as {\it cbs blow-up sequence}),
    if the $C_i\subset Y_i$ are all complete bundle-sections.

  We consider the following.

\begin{question}\label{cibu.ques}
  Let $Z\subset Y$ be smooth projective  varieties.
  Is there a  cbs blow-up sequence  $Y_r\to\cdots\to Y_0:=Y$
  with centers $C_i\subset Y_i$ such that $\dim C_i<\dim Z$ for every $i$, and
  $Z_r\subset Y_r$ is an irreducible component (that is, a connected component) of a smooth, complete bundle-section?
\end{question}

Most likely the answer is yes, but we prove this only when
$\dim Z<\frac14 \dim Y$; this is sufficient for our purposes, as explained in the introduction.

\begin{theo}\label{cibu.14.thm}
  Let $Z\subset Y$ be smooth projective  varieties such that $\dim Z<\frac14 \dim Y$.
  Then there is a  cbs blow-up sequence  $Y_r\to\cdots\to Y_0:=Y$
  with centers $C_i\subset Y_i$, such that
    $\dim C_i<\dim Z$ for every $i$, and  $Z_r\subset Y_r$ is a connected component of a smooth complete bundle-section  $Z_r^*\subset Y_r$.
\end{theo}

We will prove this theorem as an almost immediate consequence of  Property~${\cbs}_d$ stated in \ref{cbs.1.prop}.
In the inductive proof of Theorem~\ref{cibu.14.thm}  we need a stronger version, where the centers $C_i$ are in `general position'
with respect to some other subvarieties.
To understand what we need,
 consider the  blow-up of
$H:=(xy+z^2)\subset \a^4$ along the line $L:=(x=z=t=0)$. In one chart we get the equation $H'=(x_1y+z_1^2t_1=0)$. Thus $H'$ does not have ordinary double points.   Here $L\subset H$, and it is transversal to the singular set of $H$, which is $(x=y=z=0)$.

This leads to the following definition.

\begin{Defi} {\rm (Full intersection property) \label{fip.say}
  Let
  $Z\subset Y$ be  schemes.
  A closed subset $U\subset Y$ has {\it full intersection} with $Z$, if
  $Z\cap U$ is a union of connected components of $U$.
  A blow-up sequence  $Y_r\to \cdots \to Y_0=Y$ has
  {\it full intersection} with $Z$ if
the birational transforms  $Z_i\subset Y_i$ are defined, and
  each blow-up center  $C_i$ has
  full intersection with  $Z_i$.}
\end{Defi}

Let $Y$ be a smooth variety and $Z,C$ smooth  subvarieties. We will say that $Z$ has normal crossings with $C$, if  the intersection $Z\cap C$ is smooth.

\begin{lemm} (Elementary blow-up lemmas) \label{bu.lems}
   Let $\pi:Y':=B_CY\to Y$ be the blow-up and $Z'\subset Y'$ the birational transform.

\medskip

  (i) If $Z$ has normal crossings with $C$, then $Z'$ is smooth.

(ii) If $Z$ is a complete bundle-section in $Y$ and $C$ has full intersection with $Z$, then $Z'$ is a complete bundle-section in $Y'$.

  \medskip

  In addition, let  $H\subset Y$ be a hypersurface that
  has only  ordinary double points along some smooth
  $D\subsetneq H$.

\medskip

   (iii) If $C=D$, then  $H'$ is smooth.

  (iv) If $C$ has full intersection with $D$ and normal crossings with $H\setminus D$,  then $H'$ has only ordinary double points   along the proper transform $D'$ of $D$.
\end{lemm}
\begin{proof} We only prove (iv), as the other statements are completely standard. Using (i), we only have to check what happens over the components of $C$ contained in $D$. Let $r:={\rm dim}\,C$, $d:={\rm dim}\,D$, so $r\leq d$, and $n:={\rm dim}\,Y={\rm dim}\,H+1$. By assumption, we can construct  local analytic coordinates $z_1,\ldots,\,z_n$ on $Y$ such that $D$ is defined by $z_i=0,\,i\geq d+1$, $C$ is defined by $z_i=0,\,i\geq r+1$, and $H$ is defined by
\begin{eqnarray}\label{eqpourlebuexpl} f(z):=\sum_{i=d+1}^{n}z_i^2=0.\end{eqnarray}
The blow-up $\widetilde{Y}$ of $Y$ along $C$ is defined inside
$$\mathbb{P}^{n-r-1}\times Y$$ by the equations
$$Y_iz_j=Y_jz_i$$
for $i,\,j\geq r+1$, the $Y_i$'s being homogeneous coordinates on $\mathbb{P}^{n-r-1}$.
On the open set where $Y_k\not=0$, we have local coordinates
$$ z_l,\,\,\,\,1\leq l\leq r,\,\,\,\,\,\,\,\,y_j,\, \,\,j\geq r+1,\,j\not= k,\,\,\,\,\,\,\,{\rm and}\,\,z_k$$
on $\widetilde{Y}$ and the blow-up map is given by
$$ z_j=z_k y_j,\,\,\,{\rm for}\,\,j\geq r+1,\,j\not= k,$$
the exceptional divisor being defined by $z_k=0$.
We examine separately  the two   cases $k\leq d$  and  $ d+1\leq k\leq n$.

If $k\leq d$, the local equations for $\widetilde{H}$ is
\begin{eqnarray}\label{eqpourlebuexplcase1} \sum_{i=d+1}^{n}y_i^2=0,\end{eqnarray}
and $\widetilde{D}$ is defined by $y_i=0$ for any $i\geq d+1$.
It follows that $\widetilde{H}$ has  ordinary quadratic singularities along $\widetilde{D}$.

If $d+1\leq k\leq n$, the local equation for $\widetilde{H}$ is
\begin{eqnarray}\label{eqpourlebuexplcase1} 1+\sum_{i=d+1,i\not=k}^{n}y_i^2=0,\end{eqnarray}
and $\widetilde{D}$ is defined by $y_i=0$ for any $i\geq d+1$.
It follows that $\widetilde{H}$ is smooth in this open set.
\end{proof}
We also need the following subtler variant.  This is the main point in the proof where going from  complete intersections to complete bundle-sections becomes necessary.

\begin{lemm} \label{subtle.bu.lem} Let $Y$ be smooth projective and let $M\subset Y$ be a complete bundle-section. Assume that $M$ has only  ordinary double points along some smooth subvariety
   $D\subsetneq M$. Let $\pi:Y':=B_DY\to Y$ be the blow-up of $Y$ along $D$ and $M'$ the  birational transform of $M$ in $Y'$. Then  $M'\subset Y'$ is a smooth complete bundle-section.
   \end{lemm}
\begin{proof}
 We know  that $M\subset Y$ is the zero-set of a transverse section $s$ of a vector bundle $\mathcal{F}$ on $Y$. Let $E$ be the exceptional divisor of the blow-up map $\pi:Y'\rightarrow Y$.  We construct  a vector bundle $\mathcal{F}'$ on  $Y'$ by modifying $\pi^*\mathcal{F}(-E)$ along $E$. As $D\subset M$, the section $s$ vanishes along $D$ and has a differential
$$ds: N_{D/Y}\rightarrow \mathcal{F}_{\mid D}.$$
This differential has corank $1$, as follows from the fact that $M$ is singular with hypersurface singularities along $D$. We thus have a quotient line bundle $\mathcal{L}$ of $\mathcal{F}_{\mid D}$, and denoting $\pi_{E}:E\rightarrow D$ the restriction of $\pi $ to $E$, we get a quotient map constructed as the composition
$$q:\pi^*\mathcal{F}(-E)\rightarrow \pi_{E}^*\mathcal{F}_{\mid D}(-E)\rightarrow \pi_{E}^*\mathcal{L}(-E).$$

We set

$${\mathcal{F}'}:={\rm Ker}\,q.$$
This is a vector bundle  on $Y'$. Furthermore, we observe that, by construction, the section $\pi^*s$ of $\pi^*\mathcal{F}$ provides  a section ${s'}$ of ${\mathcal{F}'}\subset \pi^*\mathcal{F}$. One then  checks, using the fact  that the  singularities of $M$ are ordinary quadratic along $D$, that the vanishing locus of ${s'}$ is exactly the proper transform $M'$.
\end{proof}

\medskip

We can now state the inductive forms of
Theorem~\ref{cibu.14.thm}. Consider the following statements \ref{cbs.1.prop} and \ref{cbs.2.prop} depending on dimension $d$.

\begin{say} {\bf Property ${\cbs}_d$}.\label{cbs.1.prop}
     Let $Z\subsetneq X\subset Y$ be smooth, projective  varieties, $\dim Z\leq d$ and $\dim Y>4d$.  Assume that $X\subset Y$ is a smooth complete bundle-section, and
  $\dim X<\frac12 \dim Y$. Let $Z\subset W^j\subset Y$ be  smooth subvarieties such that  $\dim W^j<\frac12 \dim Y$.

  Then there is a cbs blow-up sequence  $\Pi:Y_r\to\cdots\to Y_0:=Y$
  with centers $C_i\subset Y_i$, such that
  \begin{enumerate}
  \item   $\dim C_i<\dim Z$ for every $i$,
  \item  each $C_i$ has full intersection with $Z_i, X_i$ and the  $W^j_i$, and
  \item   $Z_r$ is a union of  irreducible components of a smooth complete bundle-section  $Z_r^*\subset X_r\subset Y_r$.
  \end{enumerate}
      \end{say}
\begin{rema}\label{remanewlabel} {\rm
    Note that we do not claim that $Z_r^*$ has
    full intersection with  the  $W^j_r$.
    In our construction,  $Z_r^*$ is essentially the birational transform of a   complete intersection $Z^*$ of the same dimension as $Z$, and  that contains  $Z$ as an irreducible component.
      We can thus guarantee that   $Z_r^*$ is in general position away from
      $\Pi^{-1}(Z)$. Note, however, that the blow-up sequence depends on  $Z^*$ in a complicated way,
      so it is unlikely that  we can guarantee that $Z_r^*$ is also in general position along  $\Pi^{-1}(Z)\setminus Z_r$. This will cause some difficulties in the proof below.}
      \end{rema}

\begin{say}{\bf Property ${\cbs}'_d$.} \label{cbs.2.prop}  Let $Z\subset X\subset Y$ be smooth, projective  varieties, with $\dim Z\leq d$  and $\dim Y>4d$.  Assume that $X\subset Y$ is a smooth complete bundle-section and that
  $\dim X<\frac12 \dim Y$, ${\rm dim}\,Z<{\rm dim}\,X$.
   Let $Z\subset W^j\subset Y$ be smooth  subvarieties such that  $\dim W^j<\frac12 \dim Y$.

  Then there is a cbs blow-up sequence  $Y_r\to\cdots\to Y_0:=Y$
  with centers $C_i\subset Y_i$, such that
  \begin{enumerate}
  \item   $\dim C_i<\dim Z$ for every $i$,
  \item  each $C_i$ has full intersection with $Z_i, X_i$ and the  $W^j_i$, and
  \item   there is a smooth complete bundle-section $X_{r}^{(1)}\subset Y_r$, such that
    $Z_{r}\subset X_{r}^{(1)}\subset X_{r}$ and
    $\dim X_{r}^{(1)}<\dim X_r$.
  \end{enumerate}
\end{say}
\begin{rema}{\rm    In the construction below, $X_{r}^{(1)}$ is  a subset of $X_r$ of codimension 1. If $\dim X\geq \dim Z+2$, then  every irreducible component of $X_r$ contains a unique irreducible component of $X_{r}^{(1)}$.
  If $\dim X=\dim Z+1$, then  $Z_{r}$ is a union of  irreducible components of $X_{r}^{(1)}$, but usually there are other irreducible components as well.}
\end{rema}

\begin{theo}\label{cbs.cbs.thm}
  ${\cbs}_d$ and ${\cbs}'_d$ hold for every $d$.
\end{theo}

For the proof,   we use  induction on $d$, and  show that
$
{\cbs}_{d-1} \Rightarrow {\cbs}'_d  \Rightarrow {\cbs}_d.
$
Note that ${\cbs}_0$ is clear.

\medskip
\begin{proof}[Proof of  ${\cbs}'_d  \Rightarrow {\cbs}_d$]
 Since $\dim Z<\frac14 \dim Y$, there exists  by Lemma \ref{letroptropfacile} a smooth complete intersection $X\subset Y$ containing $Z$, such that $\dim X<\frac12 \dim Y$.

  We are done if $\dim Z=\dim X$. Otherwise,
  using Property ${\cbs}'_d$  for $d=\dim Z$,
    there is a
smooth cbs blow-up sequence
  $Y_{r_1}\to \cdots \to  Y_0= Y$, whose centers
have full intersections with $Z, X$, and a
smooth complete bundle-section $X_{r_1}^{(1)}\subset Y_{r_1}$ such that
$Z_{r_1}\subset X_{r_1}^{(1)}\subset X_{r_1}\subset Y_{r_1}$ and
$\dim X_{r}^{(1)}<\dim X_r$.

We now  replace  $Z\subset X\subset Y$  by
$Z_r\subset X_r^{(1)}\subset Y_r$ and repeat the argument to get
$$
Z_{r_i}\subset X_{r_i}^{(i)}\subset X_{r_i}^{(i-1)}\subset Y_{r_i}, \qtq{for} i=2,\dots
$$
With each step we lower the dimension of the  smooth complete bundle-section
$X_{r_i}^{(i)}\subset Y_{r_i}$, until we reach
$$
Z_{r_m}\subset X_{r_m}^{(m)}\subset Y_{r_m},
$$
such that $\dim Z_{r_m}=\dim X_{r_m}^{(m)}$. \end{proof}

\begin{proof}[Proof of  ${\cbs}_{d-1} \Rightarrow {\cbs}'_d$]
Take a general hypersurface
$Z\subset H\subset Y$. Then  $H\cap X$  has ordinary double points along some smooth
$D\subsetneq Z$ by Lemma \ref{letroptropfacile}. We apply ${\cbs}_{d-1}$ to $D$ to get   $D_r$ which is a union of
irreducible components of a smooth cbs  $D^*_r$.
Next we should blow up $D^*_r$.  Using  Lemma \ref{bu.lems}(iii) we get that
the birational transform of $(H\cap X)_r$ is now smooth over $D_r$, and it is a  complete bundle-section  by Lemma~\ref{subtle.bu.lem}.
However, we also need to guarantee that the other components
$D^*_r\setminus D_r$  have full intersection with  $H_r$ and $X_r$.
As we noted in  Remark~\ref{remanewlabel}, this is not clear.

We go around this problem by
creating an auxiliary  general complete intersection
$\bar X \subset Y$   that contains $D$ and has   dimension $<\frac12 \dim Y$.
Using  $\dim W^j<\frac12 \dim Y$, we can achieve  that
$\bar X\cap X=D$ and
$\bar X\cap W^j=D$ for every $j$, scheme theoretically.

Now we apply ${\cbs}_{d-1}$ to $\bar D:=D\subset \bar X\subset Y$
 and
$\bar W^j:=W^j$,  with the original $X$ playing the role of a new $\bar W^0$.
We then get  $\bar D^*_r\supset \bar D_r$, which is contained in
$\bar X_r$.  In particular,
$$
\bar D^*_r\cap \bar W^j_r\subset \bar X_r\cap \bar W^j_r\subset  \bar D_r=D_r.
$$
For $j=0$ this gives that $\bar D^*_r\cap X_r\subset D_r$.
Thus
$\bar D^*_r\setminus \bar D_r$ is disjoint from $X_r$ and the $W^j_r$, hence
 $\bar D^*_r$ has full intersections with   $Z_r, X_r$ and $W^j_r$, as needed.
\end{proof}

\begin{proof}[Proof of Theorem~\ref{cibu.14.thm}]
Since $\dim Z<\frac14 \dim Y$, there is a smooth complete intersection
$Z\subset X\subset Y$ such that $\dim X<\frac12 \dim Y$.
We can now apply Property~${\cbs}_d$ stated in  \ref{cbs.1.prop}  with $W^j=\emptyset$;
the latter is shown to hold in Theorem~\ref{cbs.cbs.thm}.
\end{proof}

 \section{The case of  homogeneous  varieties \label{sechomogeneous}}
 This section is devoted to the case of cycles on homogeneous varieties, for which stronger results are available.
 \subsection{Cycles on abelian varieties \label{sectheoab}}
We start with the proof of Theorem \ref{theoab}(ii).  The result in this case  is  stronger than Theorem \ref{theo0cyclesflpsh} since it states that \begin{eqnarray}\label{eqnewpourabsm} {\rm CH}_d(A)={\rm CH}_d(A)_{\rm sm_*Ch}\end{eqnarray} for any abelian variety $A$ and any integer $d$.
 \begin{proof} [Proof of Theorem \ref{theoab}(ii)]  Let $z\in{\rm CH}_d(A)$. We want to prove that
 \begin{eqnarray}\label{eq26sept} z\in{\rm CH}_d(A)_{\rm sm_*Ch}. \end{eqnarray}
 We can assume $z=[Z]$ for some subvariety $Z\subset A$ of dimension $d$. We denote by $\tau:\widetilde{Z}\rightarrow A$   a desingularization of $Z$. Consider the morphism
 $$\phi: A\times  \widetilde{Z}\rightarrow A$$
 $$ (x,\tilde{z})\mapsto x+\tau(\tilde{z}).$$
 Obviously $\phi$ is smooth, since it is $A$-equivariant. Furthermore, we have, denoting $0_A\in A$ the origin
\begin{eqnarray}\label{eqpourzab} z=\phi_*([0_A\times \widetilde{Z}]).
\end{eqnarray}
If $[0_A]\in{\rm CH}_0(A)$ belongs to the subring ${\rm CH}^*(A)_{\rm Ch}$ of ${\rm CH}^*(A)$ which is generated by Chern classes of coherent sheaves  on $A$, so does ${\rm pr}_1^*([0_A])=[0_A\times \widetilde{Z}]\in {\rm CH}(A\times  \widetilde{Z})$, hence (\ref{eqpourzab}) implies (\ref{eq26sept}) in this case.  It is proved however  in \cite{debarre}  that for a very general abelian variety $A$ with sufficiently divisible polarization degree and high dimension, the class of a point does not belong to ${\rm CH}^*(A)_{\rm Ch}$, so we cannot apply the argument directly to $A$.  Nevertheless, Debarre also proves in {\it loc. cit.}  that, if $J$ is the Jacobian of a curve $C$ of genus $g$, then for  any point $x$ of $J$, there exists a rank $g$ vector bundle on $J$ with a section whose zero locus is  $\{x\}$ (with its reduced structure).  In particular, the class $[x]$  belongs to  ${\rm CH}^*(J)_{\rm Ch}$ (a result that was also proved by Mattuck in \cite{mattuck}).
Let now $j:C\hookrightarrow  A$ be the inclusion of a smooth curve of genus $g$ which is a complete intersection of ample hypersurfaces in $A$. Then by Lefschetz theorem on hyperplane sections, we have a surjective (hence smooth) morphism $\psi=j_*:J=JC\rightarrow A$ of abelian varieties, and $\psi_*([0_J])=[0_A]$. Let
$$\phi_J:J\times \widetilde{Z}\rightarrow A$$
be the composite $\phi\circ (\psi,Id)$. Then $\phi_J$ is smooth and we have
\begin{eqnarray}\label{eqpourzab2} z=\phi_{J*}([0_J\times \widetilde{Z}]).
\end{eqnarray}
As $[0_J]$ belongs to  ${\rm CH}^*(J)_{\rm Ch}$, $[0_J\times \widetilde{Z}]$ belongs to  ${\rm CH}^*(J\times \widetilde{Z})_{\rm Ch}$, so (\ref{eq26sept}) follows from  formula (\ref{eqpourzab2}).
 \end{proof}
 We easily get the following consequence (we refer to the introduction for the definition of ``strongly smoothable"):

\begin{theo}\label{theopourabstrogn} (i) Let $A$ be an abelian variety of dimension $g$. Then for any integer $d$ such that $2d<g$, cycles $z\in {\rm CH}_d(A)$ are strongly smoothable.

(ii) Let $X$ be a smooth projective variety of dimension $n$ and $W_i\subset X$ be a finite set of smooth subvarieties. Then if $2d<n$, any cycle $z\in {\rm CH}_d(X)$ is an integral combination of classes of smooth subvarieties which intersect all $W_i$ in a proper way.
\end{theo}
Statement (i) follows from the equality (\ref{eqnewpourabsm}) proved above and from the following
\begin{lemm}\label{lesupstrong} Let $X$ be smooth projective of dimension $n$. Then, if $2d<n$, cycles in
${\rm CH}_d(X)_{\rm sm_*Ch}$ are strongly smoothable.
\end{lemm}
\begin{proof} A cycle $z\in {\rm CH}_d(X)_{\rm sm_*Ch}$ is of the form
$$z=f_* z'$$
where $f:Y\rightarrow X$ is smooth projective and $z'$ belongs to the subring of ${\rm CH}^*(Y)$ generated by divisor classes. Let $W_i\subset X$ be a finite number of smooth subvarieties of codimension $c_i$. As $f$ is smooth, the inverse images $f^{-1}(W_i)$ are smooth of  codimension $c_i$.
The  cycle $z'$ is a combination with integral coefficients of classes of  general complete intersections $Z''$ of very ample hypersurfaces, which are smooth  in general position and such that their intersections $Z''\cap f^{-1}(W_i)$ are  also smooth  in general position, and of dimension $d-c_i$. If $2d <n$, then $2(d-c_i)< n-c_i={\rm dim}\,W_i$. Hence Proposition \ref{letranswithflat} applies to the smooth morphisms $f_{\mid f^{-1}(W_i)}\rightarrow W_i$, showing that $f(Z''\cap f^{-1}(W_i))=f(Z'')\cap W_i$ is smooth of dimension $d-c_i$.
\end{proof}

The proof of statement (ii) follows from the equality (\ref{eqzeroflpsh}) and the following
\begin{lemm}\label{lastle} Let $Y,\,X$ be smooth projective and $f:Y\rightarrow X$ be a flat morphism, and let $z=f_* z'$,
where  $z'$ belongs to the subring of ${\rm CH}^*(Y)$ generated by divisor classes. Let $W_i\subset X$ be a finite number of  subvarieties of codimension $c_i$. Then, if $2d<n$, $z$ is represented by a  cycle of smooth subvarieties  which intersect each $W_i$ in a proper way.
\end{lemm}
Lemma \ref{lastle} is proved exactly as above, using the fact that, by flatness of $f$, the varieties $f^{-1}(W_i)\subset Y$ have codimension $c_i$.
 \subsection{More general homogeneous varieties \label{sectheohomog}}
 We prove in this section Theorem \ref{theoab}(i), which is the following statement
 \begin{theo} \label{theohomog} Let $X$ be a homogeneous variety under a group $G$.
  If  there exists a smooth projective $G$-equivariant completion $\overline{G}$ of $G$ which satisfies ${\rm CH}_0(\overline{G})= {\rm CH}_0(\overline{G})_{\rm sm_*Ch}$, then
 \begin{eqnarray}\label{eqnewdupourth2}  {\rm CH }_d(X)_{\rm sm_*Ch}={\rm CH}_d(X)
 \end{eqnarray}
 for all $d$.
  \end{theo}

 \begin{proof}[Proof of Theorem \ref{theohomog}]   Let $z\in {\rm CH}_d(X)$ be the class of a subvariety $Z\subset X$ and let $\tau:\widetilde{Z}\rightarrow X$ be a desingularization of $Z$. We consider the morphism
 $$ f: G\times \widetilde{Z}\rightarrow X,$$
 $$(g,\tilde{z})\mapsto g\cdot \tau(\tilde{z}).$$
 The morphism $f$ is obviously smooth  since it is $G$-equivariant. It is however not proper, but it provides  a $G$-equivariant rational map
 \begin{eqnarray}\label{eqmorphet} F:\overline{G}\times \widetilde{Z}\dashrightarrow X,\end{eqnarray}
  where $\overline{G}$ is any smooth projective completion of $G$ on which $G$ acts (which exists by $G$-equivariant resolution of singularities \cite{kollarreso}).  The action of $G$ on the left  hand side of (\ref{eqmorphet}) is via its action on  $\overline{G}$.
 By $G$-equivariant resolution of indeterminacies \cite{reichtein}, there exist a smooth projective variety $Y$ on which $G$ acts, a $G$-equivariant birational morphism $\eta: Y\rightarrow \overline{G}\times \widetilde{Z}$, and a morphism
 $$\widetilde{F}: Y\rightarrow X,$$
 such that $F\circ \eta=\widetilde{F}$ as rational maps to $X$. The morphism $\widetilde{F}$ is proper since $Y$ is projective, and it is again smooth  because it is $G$-equivariant.  Furthermore we have
 \begin{eqnarray} \label{eqdu26aout} z= \widetilde{F}_*( \eta^*([e\times \widetilde{Z}]))\,\,{\rm in}\,\,{\rm CH}_d(X),
 \end{eqnarray}
 where $e\in G\subset \overline{G}$ is the neutral element. We choose now $\overline{G}$ as in Theorem \ref{theohomog}. Then there exist
 a smooth projective variety $W$ (nonnecessarily connected but that we can assume  equidimensional)   and a smooth  proper morphism $\phi: W\rightarrow  \overline{G}$ such that $e$  can be written as
 \begin{eqnarray} \label{eqpouredu 2-aout}  e=\phi_*(D_1\cdot\ldots \cdot D_N)\,\,{\rm in} \,\,{\rm CH}_0(\overline{G}),\,\,N={\rm dim}\,W,
 \end{eqnarray}
 for some divisors $D_l\in {\rm CH}^1(W)$.
 Let $Y':=W\times_{\overline{G}} Y$, with first projection $q: Y'\rightarrow W$ and  second projection $p: Y'\rightarrow Y$. Then $Y'$ is smooth projective and $p:Y'\rightarrow Y$ is smooth. Denoting $\widetilde{F}':= \widetilde{F}\circ p: Y'\rightarrow X$, $\widetilde{F}'$ is  also smooth. Moreover, by (\ref{eqdu26aout}) and (\ref{eqpouredu 2-aout}), $z$  can be written as
 \begin{eqnarray}\label{eqzequal}  z=\widetilde{F}'_*(  q^*D_{1}\cdot \ldots \cdot  q^*D_{N})\,\,{\rm in}\,\,{\rm CH}_d(X).\end{eqnarray}

 Formula  (\ref{eqzequal}) shows that
 $z\in {\rm CH}_d(X)_{\rm sm_*Ch}$.
 \end{proof}

    \end{document}